\documentclass[11pt,a4paper]{article}

\usepackage[leqno]{amsmath}
\usepackage{amssymb,amsthm,upref,amscd}
\usepackage[T1]{fontenc}
\usepackage{times}
\usepackage{cancel}
\usepackage{cite}
\usepackage{amsfonts}

\newtheorem{Thm}{Theorem}[section]

\newtheorem{Lem}[Thm]{Lemma}

\newtheorem{Prop}[Thm]{Proposition}

\newtheorem{Rem}[Thm]{Remark}

\numberwithin{equation}{section}

\newcommand{\R}{\mathbb{R}}

\newcommand{\N}{\mathbb{N}}

\newcommand{\cC}{{\mathcal C}}
\newcommand{\cD}{{\mathcal D}}

\newcommand{\cF}{{\mathcal F}}

\newcommand{\cL}{{\mathcal L}}
\newcommand{\cM}{{\mathcal M}}
\newcommand{\cN}{{\mathcal N}}

\newcommand{\al}{\alpha}

\newcommand{\ga}{\gamma}
\newcommand{\de}{\delta}
\newcommand{\la}{\lambda}

\newcommand{\De}{\Delta}
\newcommand{\Ga}{\Gamma}
\newcommand{\Om}{\Omega}
\newcommand{\om}{\omega}
\newcommand{\vphi}{\varphi}
\newcommand{\eps}{\varepsilon}
\newcommand{\si}{\sigma}
\newcommand{\Si}{\Sigma}
\newcommand{\pa}{\partial}
\newcommand{\wt}{\widetilde}

\newcommand{\cat}{\text{\rm cat}}
\newcommand{\vol}{\text{\rm vol}}

\newcommand{\id}{\text{\rm id}}
\newcommand{\sign}{\text{\rm sign\,}}

\newcommand{\weakto}{\rightharpoonup}

\newcommand\mytop[2]{\genfrac{}{}{0pt}{}{#1}{#2}}

\newcommand{\beq[1]}{\begin{equation}\label{eq:#1}}
\newcommand{\eeq}{\end{equation}}

\newenvironment{altproof}[1]
{\noindent
{\em Proof of {#1}}.}
{\nopagebreak\mbox{}\hfill $\Box$\par\addvspace{0.5cm}}

\begin{document}

\bibliographystyle{plain}

\title{Equilibria of vortex type Hamiltonians on closed surfaces}

\author{\sc{Mohameden Ahmedou}\footnote{Supported by DFG grant AH 156/2-1.},
\sc{Thomas Bartsch}\footnote{Supported by DFG grant BA 1009/19-1.} ,
\sc{Tim Fiernkranz}
}

\date{}
\maketitle
\begin{abstract}
 We prove the existence of critical points of vortex type Hamiltonians
 \[
   H(p_1,\ldots, p_N)
     = \sum_{\mytop{i,j=1}{i\ne j}}^N \Ga_i\Ga_jG(p_i,p_j)+\Psi(p_1,\dots,p_N)
 \]
on a closed Riemannian surface $(\Si,g)$ which is not homeomorphic to the sphere or the projective plane. Here $G$ denotes the Green function of the Laplace-Beltrami operator in $\Si$, $\Psi:\Si^N\to\R$ may be any function of class $\cC^1$, and $\Ga_1,\dots,\Ga_N\in\R\setminus\{0\}$ are the vorticities. The Kirchhoff-Routh Hamiltonian from fluid dynamics corresponds to $\Psi(p) = -\sum_{i=1}^N \Ga_i^2h(p_i,p_i)$ where $h:\Si\times\Si\to\R$ is the regular part of the Laplace-Beltrami operator. We obtain critical points $p=(p_1,\dots,p_N)$ for arbitrary $N$ and vorticities $(\Ga_1,\dots,\Ga_N)$ in $\R^N\setminus V$ where $V$ is an explicitly given algebraic variety of codimension 1.
\end{abstract}




%
\section{Introduction}\label{sec:intro}
Let $(\Si,g)$ be a closed (compact and without boundary) Riemannian surface, and let $G$ be the Green function of the Laplace-Beltrami operator $\De_g$, i.e.\ for every $p\in\Si$ there holds:
\begin{equation}\label{eq:Green}
\left\{
\begin{aligned}
  &\De_g G(p,\cdot) = \de_p-\frac1{\vol_g(\Si)}\qquad\text{in the distributional sense}\\
  &\int_\Si G(p,\cdot)\, dv_g = 0
\end{aligned}
\right.
\end{equation}
where $dv_g$ denotes the area element of $(\Si,g)$; see \cite[\S~4.2]{Aubin:1998}. $G$ is uniquely determined by \eqref{eq:Green} and can be written as 
\[
  G(p,q) = -\frac{1}{2\pi} \log d_g(p,q) - h(p,q)
\]
with $d_g$ denoting the distance function in $\Si$, and $h:\Si\times\Si\to\R$ the regular part of $G$; see \cite[Theorem~4.13]{Aubin:1982}. The paper deals with the existence of critical points of vortex type Hamiltonians
\[
  H:\cF_N\Si = \big\{(p_1,\dots,p_N)\in\Si^N: \text{ $p_i\ne p_j$ for $i\ne j$}\big\} \to \R
\]
of the form
\begin{equation}\label{eq:ham}
  H(p_1,\ldots, p_N) = \sum_{\mytop{i,j=1}{i\ne j}}^N \Ga_i\Ga_jG(p_i,p_j) + \Psi(p_1,\dots,p_N).
\end{equation}
Here $\Psi:\Si^N\to\R$ is any $\cC^1$ function, and $\Ga_1,\dots,\Ga_N\in\R\setminus\{0\}$ are fixed, the {\it vortex strengths}. In the fluid dynamics context one has $\Psi = \Psi_{KR} = -\sum_{i=1}^N \Ga_i^2h(p_i,p_i)$. The corresponding Hamiltonian $H=H_{KR}$ is the Kirchhoff-Routh Hamiltonian.

If $\Si$ is orientable the area form $\Om_g$ induces a symplectic form on $\cF_N\Si$:
\[
  \om = \sum_{i=1}^N\Ga_i\pi_i^*(\Om_g)\qquad\text{with $\pi_i:\cF_N\Si \to \Si$ the projection onto the $i$-th factor.}
\]
The dynamics of $N$ point vortices is then described by the Hamiltonian system
\[\tag{HS}
  \dot p = X_{H} (p) \qquad\text{for $p\in\cF_N\Si$}
\]
where the Hamiltonian vector field $X_H$ on $\cF_N\Si$ is defined by
\[
  \om\big(\,\cdot\, ,X_H(p)\big) = dH(p) \qquad\text{for $p\in\cF_N\Si$.}
\]

Point vortex dynamics on planar domains or surfaces is a classical topic in fluid dynamics going back to Helmholtz \cite{Helmholtz:1858} and Kirchhoff \cite{Kirchhoff:1876}. The books \cite{Majda-Bertozzi:2002, Newton:2001, Saffman:1992} are standard modern references. Point vortex dynamics on surfaces has received a lot of attention but mostly for $\Si$ having genus zero, in particular for the round sphere, or surfaces of revolution, or being highly symmetric; see \cite{Boatto-Koiller:2014, dAprile-Esposito:2017, Dritschel-Boatto:2015, Kimura:1999, Newton:2001, Sakajo-Shimizu:2016} and the references therein. Concerning the existence of equilibria, if $\Ga_i\Ga_j>0$ for all $i,j=1,\dots,N$ then $H(p)\to\infty$ as $p\to\pa\cF_N\Si\subset\Si^N$, in particular $H$ is bounded from below. In that case $H$ has at least $\cat\big(\cF_N\Si\big)$ different critical values or infinitely many critical points by standard Lusternik-Schnirelmann theory. The problem is more difficult in the presence of counter-rotating vortices, as the results in \cite{Ba-Pistoia:2015, Ba-Pistoia-Weth:2010} for the problem on bounded domains $\Om\subset\R^2$ show. In \cite{dAprile-Esposito:2017} the authors consider $H_{KR}$ on  a closed surface and assume that there are $m$ point vortices $p_1,\dots,p_{m}$ with vortex strengths $\Ga_1,\dots,\Ga_{m}>0$, and $N-m$ point vortices $p_{m+1},\dots,p_{N}\in\Si$ with $\Ga_{m+1},\dots,\Ga_N<0$. However the vortices with negative $\Ga_i$ are located at fixed positions, so that only the $p_1,\dots,p_{m}$ move freely in $\Si\setminus\{p_{m+1},\dots,p_{N}\}$. Thus the Hamiltonian depends on  $p_1,\dots,p_{m}$, and is independent of $p_i$ for $i=m+1,\dots,N$.

Vortex type Hamiltonians appear also  in connection with blowing up phenomena arising in various partial differential equations from mathematical physics (the sinh-Poisson equation, regular and singular mean field equations, Chern-Simons equations, Toda and Liouville systems, etc.). Namely concentration points of blowing up solutions are located at equilibria of some Hamiltonian function of the above type; see 
\cite{Bartolucci-Pistoia:2007, Ba-Pistoia:2015, Ba-Pistoia-Weth:2010, dAprile:2013, dAprile:2015} for problems on domains $\Om\subset\R^2$ instead of a surface $\Si$, or \cite{chen-Lin:2003, delpino-Kowalczyk-Musso:2005, Esposito-Figueroa:2014, DelPino-et-al:2015, dAprile-Esposito:2017, Figueroa:2014} for surfaces. Furthermore they appear in the framework of critical point theory at infinity as  functional at infinity whose critical points govern the topological contribution of the  critical points at infinity to the difference of topology between the level sets of the associated Euler-Lagrange functional; see \cite{Ahmedou-BenAyed:2017, Ahmedou-BenAyed-Lucia:2017, Ahmedou-BenAyed:2021}. In these settings it makes sense to include non-orientable surfaces.

Now we state our results on the existence of point vortex equilibria.

\begin{Thm}\label{thm:main}
  Suppose $\Si$ is a closed surface not homeomorphic to the sphere $S^2$ or the real projective plane $\R P^2$. Suppose moreover that $N\ge3$ and that
  \begin{equation}\label{eq:Gamma-cond}
    \sum_{\mytop{i,j\in I}{i\ne j}}\Ga_i\Ga_j \ne0 \qquad\text{for all $I\subset\{1,\dots,N\}$ with $|I|\ge2$.}
  \end{equation}
  Then $H$ has a critical point. 
\end{Thm}


\begin{Rem}\label{rem:main}
a) The case $N=2$ is trivial because $G(p_1,p_2) \to \infty$ as $d_g(p_1,p_2) \to 0$. In that case $H$ has a minimum or a maximum, depending on the sign of $\Ga_1\Ga_2$.

b) A critical value can be described as
\[
  c^* := \sup_{\ga\simeq\ga_\Si} \min_{T^N} H\circ \ga
\]
where $T^N$ is the $N$-dimensional torus and $\ga_\Si: T^N\to\cF_N\Si$ will be explicitly given in the proof. The set $Crit(H,c^*):=\{p\in\cF_N\Si:\nabla H(p)=0,\ H(p)=c^*\}$ of critical points at the level $c^*$ is compact and stable, i.~e.\ for $\eps>0$ there exists $\de>0$ such that any $\wt H:\cF_N\Si\to\R$ with $\|\wt H-H\|_{\cC^1} < \de$ has a critical point in the $\eps$-neighborhood of $Crit(H,c^*)\subset\cF_N\Si\subset \Si^N$. In particular, $c^*$ is a stable critical value.

c) Observe that the set
\[
  V := \bigcup_{\mytop{I\subset\{1,\dots,N\}}{|I|\ge2}}\Big\{(\Ga_1,\dots,\Ga_N)\in\R^N:\sum_{\mytop{i,j\in I}{i\ne j}}\Ga_i\Ga_j =0\Big\}
\] 
is an algebraic variety of codimension $1$. Theorem~\ref{thm:main} applies for $(\Ga_1,\dots,\Ga_N)$ in the complement $\R^N \setminus V$ which is an open and dense subset of $\R^N$.


d) If $(\Si,g) = (S^2,g_{stand})$ is the round sphere and $H=H_{KR}$ is the classical Kirchhoff-Routh Hamiltonian then equilibrium configurations for three vortices have been completely classified; see \cite[Theorem~4.2.2]{Newton:2001}. A configuration $(p_1,p_2,p_3)\in\cF_3S^2$ is an equilibrium iff
\[
  \Ga_1(\Ga_2+\Ga_3)p_1 + \Ga_2(\Ga_1+\Ga_3)p_2 + \Ga_3(\Ga_1+\Ga_2)p_3 = 0.
\]
The three points lie in a plane, actually on a great circle. If $\Ga_1=\Ga_3<0<\Ga_2$ a simple calculation shows that an equilibrium exists iff $\Ga_1+2\Ga_2 < 0$. Assumption \eqref{eq:Gamma-cond} is equivalent to $\Ga_1+2\Ga_2 \ne 0$. It follows that Theorem~\ref{thm:main} does not hold for $\Si$ being homeomorphic to $S^2$. We expect that the same holds for the real projective plane
$\Si=\R P^2$.

e) If the vortices with negative $\Ga_i$ are at fixed positions as in \cite{dAprile-Esposito:2017} then there exists an equilibrium configuration for the vortices with positive $\Ga_i$ also if $\Si$ is homeomorphic to the sphere or the real projective plane. For these surfaces conditions on the $\Ga_i$ are required in \cite{dAprile-Esposito:2017} that are not needed for surfaces of higher genus. For surfaces of higher genus condition (1.3) from \cite{dAprile-Esposito:2017} corresponds to a simpler version of condition \eqref{eq:Gamma-cond} above, reflecting the fact that in \cite{dAprile-Esposito:2017} vortices with negative vorticities cannot interact with each other, and each vortex with positive vorticity can only interact with at most one vortex with negative vorticity.

f) It is possible to consider a surface $\Si$ with boundary. In that case one can combine our linking arguments with the techniques from \cite{Ba-Pistoia:2015, Ba-Pistoia-Weth:2010, Kuhl:2016} to control the gradient vector field $\nabla H$ near the boundary $\pa\cF_N\Si \subset \Si^N$. 

g) It would also be interesting to investigate surfaces with symmetries, and to find symmetric and non-symmetric configurations of equilibria. For the case of symmetric bounded planar domains such results have been obtained in \cite{Kuhl:2015}.
\end{Rem}

The case of the round sphere mentioned in Remark~\ref{rem:main} d), is of course very special. Our next result, which holds for all closed surfaces including those homeomorphic to the sphere or the projective plane, shows that a much weaker symmetry of $(\Si,g)$ is sufficient for the existence of an equilibrium with three vortices as in \cite[Theorem~4.2.2]{Newton:2001}.

\begin{Thm}\label{thm:symmetry}
  Let $N=3$, $\sign \Ga_i=(-1)^i$ and $\Ga_1\Ga_2+\Ga_1\Ga_3+\Ga_2\Ga_3 > 0$. Suppose there exists an isometric involution $\tau:\Si\to\Si$ and that $\Psi:\Si^3\to\R$ is invariant under the induced action of $\tau$ on $\Si^3$:  $\Psi\big(\tau(p_1),\tau(p_2),\tau(p_3)\big)=\Psi(p_1,p_2,p_3)$.
\begin{itemize}
\item[a)] If the fixed point set $\Si^\tau = \{p\in\Si: \tau(p)=p\}$ has a connected component $S$ diffeomorphic to $S^1$ then $H$ has a critical point $p^* \in \cF_3\Si$ with $p^*_1,p^*_2,p^*_3 \in S$.
\item[b)] If $\Ga_1 = \Ga_3$ and $\emptyset \ne \Si^\tau \ne \Si$ then $H$ has a critical point $p^* \in \cF_3\Si$ with $p^*_2 \in \Si^\tau$ and $p^*_3=\tau(p^*_1) \notin \Si^\tau$.
\end{itemize}
\end{Thm}

\begin{Rem}\label{rem:symmetry}
a) Theorem~\ref{thm:symmetry}~a) applies, for instance, to $(S^2,g)$ if the metric $g$ on $S^2$ is invariant under $\tau:S^2\to S^2$, $(x_1,x_2,x_3)\mapsto(x_1,x_2,-x_3)$. If $g$ is in addition invariant under the antipodal map $x \mapsto -x$ then it induces a metric $\bar g$ on $\R P^2$. Moreover $\tau$ induces an isometric involution $\bar\tau$ on $\R P^2$ with fixed point set $(\R P^2)^{\bar\tau} = \{[x]\in\R P^2: x_3=0\} \cup \{[e_3]\}$ so that Theorem~\ref{thm:symmetry}~a) can be applied. If in addition $\Ga_1=\Ga_3$ then Theorem~\ref{thm:symmetry}~b) applies in addition.

b) In the case of the round sphere the solutions obtained in Theorem~\ref{thm:symmetry} lie on one $O(3)$-orbit and correspond to those from \cite[Theorem~4.2.2]{Newton:2001}.

c) In the situation of Theorem~\ref{thm:symmetry}~a) one can obtain many more stationary point vortex configurations with all vortices in $S$ for arbitrary values of $N$. We refer the reader to \cite{Ba-Pistoia-Weth:2010} and in particular to \cite{Kuhl:2015} for results in this direction on a bounded planar domain. Since these are one-dimensional configurations of vortices the techniques from these papers can be applied.
\end{Rem}

In applications to singular limit problems for partial differential equations it is important that the critical points of $H$ are nondegenerate. In some applications to mean field equations $\Psi$ is of the form
\begin{equation}\label{eq:psi-1}
  \Psi(p_1,\dots,p_N) =  \sum_{i=1}^N \log K(p_i) - \sum_{i=1}^N \Ga_i^2h(p_i,p_i)
\end{equation}
or
\begin{equation}\label{eq:psi-2}
  \Psi(p_1,\dots,p_N) = \sum_{i=1}^{N_1} \log K_1(p_i) +  \sum_{i=N_1+1}^{N} \log K_2(p_i) - \sum_{i=1}^N \Ga_i^2h(p_i,p_i)
\end{equation}
where $K,K_1,K_2:\Si\to\R^+$ are of class $\cC^2$.

\begin{Prop}\label{prop:generic-K}
a) The  set
\[
  \{K\in\cC^2(\Si,\R^+):\text{$H$ as in \eqref{eq:ham} with $\Psi$ as in \eqref{eq:psi-1} is a Morse function}\}
\]
is an open and dense subset of $\cC^2(\Si,\R^+)$.

b) The  set
\[
  \{K\in\cC^2\big(\Si,(\R^+)^2\big):\text{$H$ as in \eqref{eq:ham} with $\Psi$ as in \eqref{eq:psi-2} is a Morse function}\}
\]
is an open and dense subset of $\cC^2\big(\Si,(\R^+)^2\big)$.
\end{Prop}

\begin{Rem}\label{rem:Morse}
  The proof of Proposition~\ref{prop:generic-K} can be adapted to prove that the set of $\Psi \in\cC^2(\Si^N)$ such that $H$ as in \eqref{eq:ham} is a Morse function is an open and dense subset of $\cC^2(\Si^N)$; see Remark~\ref{rem:proof Morse}.
\end{Rem}

The existence of periodic point vortex configurations on closed surfaces is relatively unexplored, except for the round sphere. This is work in progress. The results about periodic solutions of the point vortex problem near an equilibrium of the Robin function on a bounded planar domain from \cite{Ba:2016, Ba-Dai:2016, Ba-Gebhard:2017, Ba-Sacchet:2018, Gebhard:2018} can be transferred to surfaces. It would be very interesting to extend the results from \cite{Ba-Gebhard:2017, Gebhard:2018} about global connected continua of periodic solutions to the case of surfaces.

We conclude the introduction with discussing an application to the following sinh-Poisson mean field equation considered in the recent paper \cite{Figueroa:2022} by Figueroa:
\begin{equation}\label{eq:sinh-poisson}
  -\De_gu = \la_1\left(\frac{V_1e^u}{\int_\Si V_1e ^udv_g} - \frac1{\vol_g(\Si)}\right) - \la_2\tau\left(\frac{V_2e^{-\tau u}}{\int_\Si V_2e^{-\tau u}dv_g} - \frac1{\vol_g(\Si)}\right)
\end{equation}
Here $\la_1,\la_2,\tau>0$ and $V_1,V_2:\Si \to \R^+$ are of class $\cC^2$. The critical parameter values are $\la_1,\la_2\tau^2 \in 8\pi\N$. Theorem~1.1 of \cite{Figueroa:2022} yields the existence of a family of solutions $(\la_1(\de),\la_2(\de),u_\de) \in \R^+\times\R^+\times H^1(\Si)$, $\de\in(0,\de_0)$, of \eqref{eq:sinh-poisson}, provided the Hamiltonian function (see \cite[equation (1.7)]{Figueroa:2022})
\[
\begin{aligned}
  H(p_1,\ldots, p_N) &= \sum_{\mytop{i,j=1}{i\ne j}}^N \Ga_i\Ga_jG(p_i,p_j) - \sum_{i=1}^N \Ga_i^2h(p_i,p_i)\\
         &\hspace{1cm}- \frac1{4\pi}\sum_{j=1}^m\log V_1(p_j) - \frac1{4\pi\tau^2}\sum_{j=m+1}^N\log V_2(p_j)
\end{aligned}
\]
where $0\le m\le N$, $\Ga_1=\dots=\Ga_m=1$, $\Ga_{m+1}=\dots=\Ga_N=-\frac1\tau$, has a stable critical set $\cD\subset\cF_N\Si$, and such that certain sign-conditions hold for functions $A_1^*$, $A_2^*$ defined in \cite[equation (1.10)]{Figueroa:2022}. These solutions have $m$ positive and $N-m$ negative concentration points as $\de\to0$ and satisfy
\[
  \int_\Si u_\de\, dv_g=0, \quad \la_1(\de)\to 8\pi m, \quad \la_2(\de)\tau^2 \to 8\pi (N-m)
\]
and
\[
  \frac{\la_1(\de)V_1e^{u_\de}}{\int_\Si V_1e^{u_\de}dv_g} \weakto 8\pi\sum_{i=1}^{m}\de_{p^*_i}
  \qquad\text{and}\qquad
  \frac{\la_2(\de)\tau^2 V_2e^{-\tau u_\de}}{\int_\Si V_2e^{-\tau u_\de}dv_g} \weakto 8\pi\sum_{i=m+1}^{N}\de_{p^*_i}
\]
weakly in the sense of measures as $\de \to 0$, for some $p^*\in\cD$. The Hamiltonian has the form \eqref{eq:ham} with $\Psi$ as in \eqref{eq:psi-2}. Proposition~\ref{prop:generic-K} yields that $H$ is a Morse function for a generic choice of $V_1,V_2$, i.e.\ for $V_1,V_2$ in an open and dense subset of $\cC^2(\Si,\R^+)$. For such a choice of $V_1,V_2$, and if \eqref{eq:Gamma-cond} holds, Theorems~\ref{thm:main} and \ref{thm:symmetry} yield a critical point $p^*$ of $H$ which is non-degenerate, hence stable. For a generic choice of $V_1,V_2$ one has $A_1^*(p^*) \ne 0 \ne A_2^*(p^*)$. Theorem~1.1 of \cite{Figueroa:2022} applies if these have the same sign. In that case it is also determined whether $\la_1(\de) > 8\pi N_1$ or $\la_1(\de)< 8\pi N_1$, and similarly for $\la_2(\de)$. The condition $\sign\big(A_1^*(p^*)\big) = \sign\big(A_2^*(p^*)\big)$ is an open condition but not a generic one, and needs to be considered in addition to the existence of $p^*$. Concerning condition \eqref{eq:Gamma-cond} the equation $ \sum_{\mytop{i,j\in I}{i\ne j}}\Ga_i\Ga_j = 0$ is a quadratic equation in $\frac1\tau$, hence has at most two solutions. Therefore \eqref{eq:Gamma-cond} is generically true.

\section{The linking for Theorem~\ref{thm:main}}\label{sec:linking}
Setting $T:=[0,1]/\{0,1\}$ in this section we construct a map $\ga_\Si:T^N\to\cF_N\Si$ and a subset $\cL_\Si\subset\cF_N\Si$ such that
\begin{equation}\label{eq:bound}
  \text{$H$ is bounded on $\cL_\Si$}
\end{equation}
and
\begin{equation}\label{eq:linking}
  \ga\simeq\ga_\Si\quad\Longrightarrow\quad \ga\big(T^N\big)\cap\cL_\Si \ne \emptyset.
\end{equation}
Thus $\ga_\Si$ and $\cL_\Si$ link in $\cF_N\Si$. We consider three cases depending on whether $\Si$ is a torus, a Klein bottle or has higher genus. Let us begin with the simplest case of $\Si$ being homeomorphic to a torus. We may assume $\Si=T^2$. Choose $0<\tau_1<\dots<\tau_N<1$ and define
\[
  \ga_{T^2}:T^N\to\cF_NT^2,\quad \ga_{T^2}\big([s_1,\dots,s_N]\big) := \big([s_1,\tau_1],\dots,[s_N,\tau_N]\big).
\]
Moreover we choose $0<\si_1<\dots<\si_N<1$ and set
\[
  \cL_{T^2} := \big\{\big([\si_1,t_1],\dots,[\si_N,t_N]\big): t_1,\dots,t_N\in[0,1]\big\} \subset \cF_NT^2.
\]
Obviously $\cL_{T^2}$ is compact so that
\begin{equation}\label{eq:bound-torus}
  \min\{d_g(p_i,p_j): (p_1,\dots,p_N)\in\cL_{T^2},\ i\ne j\} > 0.
\end{equation}
This implies that $H$ is bounded on $\cL_{T^2}$.

\begin{Lem}\label{lem:torus}
If $\ga:T^N\to\cF_NT^2$ is homotopic to $\ga_{T^2}$ then $\ga\big(T^N\big)\cap\cL_{T^2}\ne\emptyset$.
\end{Lem}

\begin{proof}
We consider the map
\[
  \mu:\cF_NT^2\to T^N,\quad \mu\big([s_1,t_1],\dots,[s_N,t_N]\big)  =  [s_1,\dots,s_N] .
\]
Clearly $\mu\circ\ga_{T^2}=\id:T^N\to T^N$, hence $\deg(\mu\circ\ga_{T^2})=1$. Since $\ga$ is homotopic to $\ga_{T^2}$ we obtain $\deg(\mu\circ\ga)=1$, hence $\mu\circ\ga$ must be onto. It follows that $\mu\circ\ga(p)=[\si_1,\dots,\si_N]$ for some $p\in T^N$ which implies $\ga(p)\in\cL_{T^2}$.
\end{proof}

Next we consider the case of the Klein bottle. Here we may assume
\[
  \Si=K:=[0,1]^2/\sim\quad\text{where $(s,1)\sim(1-s,0)$, $(1,t)\sim(0,t)$.}
\]
Recall that $K$ is double-covered by $T^2$. In our notation a covering map is given by
\[
  \pi:T^2\to K,\quad \pi\big([s,t]\big):=\begin{cases}[s,2t]&\text{for $0\le t\le\frac12$}\\ [1-s,2t-1]&\text{for $\frac12\le t\le1$.}\end{cases}
\]
Setting
\[
  \pi^N(p_1,\dots,p_N) := \big(\pi(p_1),\dots,\pi(p_N)\big)
\]
we obtain a $2^N$-sheeted covering map
\[
  \pi^N: E:=\{p\in\cF_NT^2:\pi^N(p)\in\cF_NK\} \to\cF_NK.
\]
We consider $\cL_{T^2}$ as above but with $0<\si_1<\dots<\si_N<\frac12$ so that $\cL_{T^2}\subset E$, and define $\cL_K:=\pi^N\big(\cL_{T^2}\big)$. Then $\cL_K$ is compact and \eqref{eq:bound-torus} holds with $\cL_K$ instead of $\cL_{T^2}$. Thus $H$ is bounded on $\cL_K$ as in the case of the torus. Next we choose $0<\tau_1<\dots<\tau_N<\frac12$ and set $\ga_K:=\pi^N\circ\ga_{T^2}$, i.e.
\[
  \ga_K:T^N\to \cF_NK,\quad \ga_K\big([s_1,\dots,s_N]\big) = \big(\pi\big([s_1,\tau_1]\big),\dots,\pi\big([s_N,\tau_N]\big)\big).
\]
Observe that $\ga_{T^2}\big(T^N\big)\subset E\subset\cF_NT^2$ because $0<\tau_1<\dots<\tau_N<\frac12$.

\begin{Lem}\label{lem:Klein}
If $\ga:T^N\to\cF_NK$ is homotopic to $\ga_K$ then $\ga\big(T^N\big)\cap\cL_K\ne\emptyset$.
\end{Lem}

\begin{proof}
Let $h:T^N\times[0,1]\to \cF_NK$ be a homotopy between $h_0=h(\,\cdot\,,0)=\ga_K=\pi^N\circ\ga_{T^2}$ and $h_1=h(\,\cdot\,,1)=\ga$. Since $\pi^N:E\to\cF_NK$ is a covering map $h$ can be lifted to a homotopy
\[
  \wt{h}:T^N\times[0,1]\to E\subset\cF_NT^2
\]
such that $\wt{h}_0=\ga_{T^2}$ and $\pi^N\circ\wt{h}=h$, in particular $\pi^N\circ\wt{h}_1=\ga$. Now Lemma~\ref{lem:torus} yields that $\wt{h}_1\big(T^N\big)\cap\cL_{T^2}\ne\emptyset$, hence $h_1\big(T^N\big)\cap\cL_K = \pi^N\circ\wt{h}\big(T^N\big)\cap\pi^N\big(\cL_{T^2}\big) \ne \emptyset$.
\end{proof}

It remains to consider the case of a surface $\Si$ of higher genus. Here we may assume that $\Si=T^2\#\Si'$ is the connected sum of $T^2$ and a surface $\Si'$. We choose $r>0$, a point $(s_0,t_0)\in (0,1)^2 \subset T^2$, as well as $0<\si_1<\dots<\si_N<1$ and $0<\tau_1<\dots<\tau_N<1$ so that
\[
  \overline{B_{2r}(s_0,t_0)} \subset \big\{(s,t)\in (0,1)^2\subset T^2:\text{$s\ne\si_i$, $t\ne\tau_i$ for all $i=1,\dots,N$}\big\}.
\]
We may assume that $\Si'$ is attached to $T^2$ along the boundary $\pa B_r(s_0,t_0)$ of $B_r(s_0,t_0)\subset (0,1)^2\subset T^2$. Then $T^2\setminus B_r(s_0,t_0) \subset \Si$, and we have a canonical projection $\Si\to T^2/\overline{B_r(s_0,t_0)}$ identifying all points outside of $T^2\setminus B_r(s_0,t_0)$ to a point. By Tietze-Urysohn, or a simple direct construction, the maps
\[
  [0,1]^2\setminus B_{2r}(s_0,t_0) \to [0,1],\quad (s,t) \mapsto s,
\]
and
\[
  \overline{B_r(s_0,t_0)} \to [0,1],\quad (s,t)\mapsto s_0,
\]
can be extended to a continuous map $[0,1]^2\to [0,1]$. By construction this map induces a continuous map $T^2/\overline{B_r(s_0,t_0)} \to T$. Composed with the canonical projection $\Si\to T^2/\overline{B_r(s_0,t_0)}$ we obtain a continous map $\rho:\Si\to T^2/\overline{B_r(s_0,t_0)} \to T$ which satisfies $\rho([s,t])=[s]$ for $[s,t]\in T^2\setminus B_{2r}(s_0,t_0) \subset\Si$. Now we define $\ga_\Si:T^N\to\cF_N\Si$ by
\[
  \ga_\Si\big([s_1,\dots,s_N]\big) := \big([s_1,\tau_1],\dots,[s_N,\tau_N]\big) \in \cF_N\big(T^2\setminus B_{2r}(s_0,t_0)\big) \subset \cF_N\Si
\]
and
\[
  \cL_\Si := \big\{\big([\si_1,t_1],\dots,[\si_N,t_N]\big): t_1,\dots,t_N\in[0,1]\big\} \subset \cF_N\big(T^2\setminus B_{2r}(s_0,t_0)\big) \subset \cF_N\Si.
\]
Clearly $\cL_\Si$ is compact and \eqref{eq:bound-torus} holds with $\ga_\Si$ replacing $\ga_{T^2}$ so that $H$ is bounded on $\cL_\Si$. It remains to prove the linking property \eqref{eq:linking}.

\begin{Lem}\label{lem:surface}
If $\ga:T^N\to\cF_N\Si$ is homotopic to $\ga_\Si$ then $\ga(T^N)\cap\cL_\Si\ne\emptyset$.
\end{Lem}

\begin{proof}
The proof proceeds analogous to the one of Lemma~\ref{lem:torus}. Consider the map
\[
  \rho^N:\cF_N\Si\to T^N,\quad (p_1,\dots,p_N) \mapsto (\rho(p_1),\dots,\rho(p_N)).
\]
Clearly $\rho^N\circ\ga_\Si=\id:T^N\to T^N$, hence $\deg(\rho^N\circ\ga_\Si)=1$. Since $\ga$ is homotopic to $\ga_\Si$ we obtain $\deg(\rho^N\circ\ga)=1$, hence $\rho^N\circ\ga$ must be onto. It follows that $\rho^N\circ\ga(p)=[\si_1,\dots,\si_N]$ for some $p\in T^N$ which implies $\ga(p)\in\cL_\Si$.
\end{proof}

\section{Proofs of the critical point theorems}\label{sec:proofs}

For the proof of Theorem~\ref{thm:main} we first need some control on $H$ near $\pa\cF_N\Si\subset\Si^N$.

\begin{Lem}\label{lem:collision}
  If \eqref{eq:Gamma-cond} holds then for $p^* \in \pa\cF_N\Si$ there exists $p_0\in\Si$ and $c,\de>0$ such that $I:=\big\{i\in\{1,\dots,N\}:p^*_i=p_0\big\}$ has at least two elements and such that 
  \[
    |\nabla H(p)| \ge c\left(\sum_{i\in I} d_g(p_i,p^*_i)^2\right)^{-\frac12}\quad\text{for $p\in \cF_N\Si\cap B_\de(p^*)$.}
  \]
Here $B_\de(p^*)\subset \Si^N$ is the open $\de$-neighborhood of $p^* \in \pa\cF_N\Si\subset \Si^N$.
\end{Lem}

\begin{proof}
The existence of $p_0\in\Si$ such that $I:=\big\{i\in\{1,\dots,N\}:p^*_i=p_0\big\}$ has at least two elements follows immediately from $p^*\in\pa\cF_N(\Si)$. Since every surface is locally conformally flat there exists an open neighborhood $U\subset\Si$ of $p_0$, a conformal metric $\wt g = e^ug$, and a chart $\psi:U\to V\subset\R^2$ with $\psi(p_0)=0$ and such that the Christoffel symbols of $\wt g$ are trivial. This implies that
\[
  d_{\wt g}(p,q) = |\psi(p)-\psi(q)| \qquad\text{for $p,q\in U$.}
\]
We may assume that $u:\Si\to\R$ is globally defined so that $\wt g$ is a metric on all of $\Si$. Let $\wt G(p,q) =  -\frac1{2\pi}\log d_{\wt g}(p,q)-\wt h(p,q)$ denote a Green function for the Laplace-Beltrami operator of $(\Si,\wt g)$ with regular part $\wt h:\Si\times\Si\to\R$. By Lemma~\ref{lem:conformal-Green} below there exists $F\in\cC^\infty(\Si\times\Si)$ such that $\wt G = G+F$. Since $d_g(p_i,p_j)$ is bounded away from $0$ as $p\to p^*$ for $i\in I$, $j\notin I$ we have, with $p_I:=(p_i)_{i\in I}$,
\[
\begin{aligned}
  H(p) &= -\frac1{2\pi}\sum_{\mytop{i,j\in I}{i\ne j}}^N
	          \Ga_i\Ga_j \log d_{\wt g}(p_i,p_j) + R(p) \\
       &= -\frac1{2\pi}\sum_{\mytop{i,j\in I}{i\ne j}}
			      \Ga_i\Ga_j \log|\psi(p_i)-\psi(p_j)| +R(p) \equiv H_I(p_I) + R(p)
\end{aligned}
\]
where $R$ satisfies
\[
  \nabla_{p_i} R(p) = O(1) \qquad\text{as $p\to p^*$, $i\in I$.}
\]
The map $\psi_I(p_I):=\big(\psi(p_i)\big)_{i\in I}$ is defined if $p_i\in U$ for $i\in I$, hence if $p$ is close to $p^*$. Setting $z_I=(z_i)_{i\in I} \in (\R^2)^{I}$ and
\[
  H_\psi^I(z_I) = -\frac1{2\pi}\sum_{\mytop{i,j\in I}{i\ne j}} \Ga_i\Ga_j \log|z_i-z_j|
\]
we have $H_I=H_\psi^I\circ\psi_I$. Since $\psi_I$ is a diffeomorphism the lemma follows if there exists $c,\de>0$ such that
\[
  \big|\nabla H_\psi^I(z_I)\big| \ge  c\left(\sum_{i\in I} |z_i|^2\right)^{-\frac12}\quad\text{for $|z_I| < \de$.}
\]
This has been proved proved in \cite[Lemma~4.2]{Kuhl:2016}, using assumption \eqref{eq:Gamma-cond}.
\end{proof}

\begin{Lem}\label{lem:conformal-Green}
  Let $\wt g = e^ug$ be a conformal metric to $g$ on $\Si$ and $\wt G$ the Green's function for the Laplace-Beltrami operator of $(\Si,\wt g)$. Then there exists $F\in\cC^\infty(\Si\times\Si)$ such that $\wt G = G+F$.
\end{Lem}

\begin{proof}
Recall that for $f\in\cC^2(\Si)$ there holds
\begin{equation}\label{eq:conformal}
  \De_{\wt g}f=e^{-2u}\De_gf\qquad\text{and}\qquad \int_\Si e^{-2u}f\,dv_{\wt g} = \int_\Si f\,dv_g.
\end{equation}
The problem
\begin{equation}\label{eq:w}
  \De_gw = \frac{e^{2u}}{\vol_{\wt g}(\Si)}-\frac1{\vol_g(\Si)}
\end{equation}
has a solution $w\in\cC^\infty(\Si)$ because
\[
  \int_\Si \left(\frac{e^{2u}}{\vol_{\wt g}(\Si)}-\frac1{\vol_g(\Si)}\right)\,dv_g = 0
\]
Now we define $F:\Si\times\Si\to\R$ by 
\[
  F(p,q) = c-w(p)-w(q)\quad\text{with }\ c = \frac1{\vol_{\wt g}(\Si)}\int_\Si w\,dv_{\wt g} + \frac1{\vol_{g}(\Si)}\int_\Si w\,dv_{g}.
\]
Then we obtain for $f\in\cC^2(\Si)$ and $p\in\Si$, using \eqref{eq:Green}, \eqref{eq:conformal}, \eqref{eq:w}:
\[
\begin{aligned}
  &\int_\Si  \big(G(p,\cdot)+F(p,\cdot)\big)\De_{\wt g}f\,dv_{\wt g} = \int_\Si \big(G(p,\cdot)+c-w(p)-w\big)e^{-2u}\De_g f\,dv_{\wt g}\\
    &\hspace{1cm}= \int_\Si \big(G(p,\cdot)+c-w(p)-w\big)\De_g f\,dv_{g}\\
    &\hspace{1cm}= f(p) - \frac1{\vol_g(\Si)}\int_\Si f\,dv_g - \int_\Si \big(\De_g w)f\,dv_g\\
    &\hspace{1cm}= f(p) - \frac1{\vol_g(\Si)}\int_\Si f\,dv_g - \left(\frac1{\vol_{\wt g}(\Si)}\int_\Si e^{2u}f\,dv_g -  \frac1{\vol_g(\Si)}\int_\Si f\,dv_g\right)\\
    &\hspace{1cm}= f(p) - \frac1{\vol_{\wt g}(\Si)}\int_\Si f\,dv_{\wt g}
\end{aligned}
\]
Thus we have
\[
  \De_{\wt g}(G+F)(p,\cdot) = \de_p -\frac1{\vol_{\wt g}(\Si)}\qquad\text{in the distributional sense.}
\]
Concerning the normalization property in \eqref{eq:Green} we compute for $p\in\Si$, using again \eqref{eq:Green}, \eqref{eq:conformal}, \eqref{eq:w}, and the definition of $c$:
\[
\begin{aligned}
  &\int_\Si \big(G(p,\cdot)+F(p,\cdot)\big)\,dv_{\wt g} = \int_\Si G(p,\cdot)\,dv_{\wt g} + c\cdot\vol_{\wt g}(\Si) -w(p)\cdot\vol_{\wt g}(\Si) - \int_\Si w\,dv_{\wt g}\\
  &\hspace{1cm}= \int_\Si G(p,\cdot)\,dv_{\wt g} + c\cdot\vol_{\wt g}(\Si) -\left(\int_\Si G(p,\cdot)\De_gw\,dv_g+\frac1{\vol_g(\Si)}\int_\Si w\,dv_g\right)\vol_{\wt g}(\Si)\\
  &\hspace{2cm} - \int_\Si w\,dv_{\wt g}\\
  &\hspace{1cm}=  \int_\Si G(p,\cdot)\,dv_{\wt g} + c\cdot\vol_{\wt g}(\Si) -\int_\Si G(p,\cdot)\left(\frac{e^{2u}}{\vol_{\wt g}(\Si)}-\frac1{\vol_g(\Si)}\right)\,dv_g\cdot\vol_{\wt g}(\Si)\\
  &\hspace{2cm} -\frac{\vol_{\wt g}(\Si)}{\vol_g(\Si)}\int_\Si w\,dv_g - \int_\Si w\,dv_{\wt g}\\
  &\hspace{1cm}= c\cdot\vol_{\wt g}(\Si) -\frac{\vol_{\wt g}(\Si)}{\vol_g(\Si)}\int_\Si w\,dv_g - \int_\Si w\,dv_{\wt g}\\
  &\hspace{1cm}= 0.
\end{aligned}
\]
Therefore $G+F$ satisfies \eqref{eq:Green} with $\wt g$ instead of $g$, hence $\wt G=G+F$. 
\end{proof}

Let $\vphi$ be the gradient flow of $H$, and let $T^+(p)\in(0,\infty]$ be the maximal life time of $p$, i.e.\ $\vphi(p,t)$ is defined for $t\in\big[0,T^+(p)\big)$. 

\begin{Lem}\label{lem:flow limit}
  If \ $\lim_{t\to T^+(p)} H\big(\vphi(p,t)\big) <\infty$ for some $p\in\cF_N(\Si)$ then $T^+(p)=\infty$. Moreover there exists a sequence $t_n\to\infty$ and $p^*\in\cF_N\Si$ with $\vphi(p,t_n)\to p^*$ and $\nabla H\big(\vphi(p,t_n)\big) \to 0$. In particular $p^*$ is a critical point of $H$.
\end{Lem}

\begin{proof}
Set $c^*:=\lim_{t\to T^+(p)} H\big(\vphi(p,t)\big)$ and suppose $T^+(p)<\infty$. Arguing as in the proof of \cite[Lemma~4.7, Lemma~4.8]{Kuhl:2016} we first observe that we have for $0<s<t<T^+(p)$:
\[
\begin{aligned}
  &d_g^N\big(\vphi(p,s),\vphi(p,t)\big) 
        \le \int_s^t\big|\nabla H\big(\vphi(p,r)\big)\big|\,dr \le \sqrt{t-s}\left(\int_s^t \big|\nabla H\big(\vphi(p,r)\big)\big|^2\,dr\right)^{\frac12}\\
  &\hspace{1cm}= \sqrt{t-s}\sqrt{H\big(\vphi(p,t)\big)-H\big(\vphi(p,s)\big)} \le \sqrt{T^+(p)-s}\cdot\sqrt{c^*-H\big(\vphi(p,s)\big)}
\end{aligned}
\]
Here $d_g^N$ denotes the induced distance on $\Si^N$. It follows that $\vphi(p,t) \to p^*\in\pa\cF_N(\Si)$. Let $p_0\in\Si$, $I\subset\{1,\dots,N\}$ and $c>0$ be as in Lemma~\ref{lem:collision}, and let $\wt G$ be a conformal metric to $g$, $\psi:(U,p_0)\to (V,0)$ be a chart of $\Si$ around $p_0$ as in the proof of Lemma~\ref{lem:collision}. Setting $z_i(t):=\psi\big(\vphi_i(p,t)\big)$ for $i\in I$, so that $z_i(t)\to\psi(p_0)=0$ as $t\to T^+(p)$,  Lemma~\ref{lem:collision} implies for $0<s<t<T^+(p)$ close to $T^+(p)$:
\[
\begin{aligned}
  H\big(\vphi(p,t)\big)-H\big(\vphi(p,s)\big) &= \int_s^t \big|\nabla H\big(\vphi(p,r)\big)\big|^2\,dr \\
      &\ge c\int_s^t \left(\sum_{i\in I} d_g\big(\vphi_i(p,r),p_0\big)^2\right)^{-\frac12}\cdot\big|\dot\vphi(p,r)\big|_g\,dr\\
      &\ge \wt c\int_s^t \left(\sum_{i\in I} d_{\wt g}\big(\vphi_i(p,r),p_0\big)^2\right)^{-\frac12}\cdot\big|\dot\vphi(p,r)\big|_{\wt g}\,dr\\
      &\ge \wt c\int_s^t\left(\sum_{i\in I} \big|z_i(r)\big|^2\right)^{-\frac12} \cdot \left(\sum_{i\in I} \big|\dot z_i(r)\big|^2\right)^{\frac12}\,dr\\
      &\ge -\wt c\int_s^t\left(\sum_{i\in I} \big|z_i(r)\big|^2\right)^{-\frac12} \cdot \frac{d}{dr}\left(\sum_{i\in I} \big|z_i(r)\big|^2\right)^{\frac12}\,dr\\\\
      &= \frac{\wt c}2 \ln\frac{\sum_{i\in I} \big|z_i(s)\big|^2}{\sum_{i\in I} \big|z_i(t)\big|^2} \to \infty \qquad\text{as $t\to T^+(p)$.}
\end{aligned}
\]
This contradicts the assumption $H\big(\vphi(p,t)\big) \to c^*$ as $t\to T^+(p)$, hence $T^+(p)=\infty$.

Now $H\big(\vphi(p,t)\big) \to c^*$ as $t\to \infty$ implies that there exists a sequence $t_n\to\infty$ such that $\nabla H\big(\vphi(p,t_n)\big) \to 0$. Passing to a subsequence we may assume that $\vphi(p,t_n) \to p^* \in \Si^N$. Then Lemma~\ref{lem:collision} implies $p^* \in \cF_N(\Si)$.
\end{proof}

\begin{altproof}{Theorem~\ref{thm:main}}
Let $\ga\simeq\ga_\Si:T^N\to \cF_N\Si$ and $\cL_\Si\subset \cF_N\Si$ be as in Section~\ref{sec:linking} so that \eqref{eq:bound} and \eqref{eq:linking} hold. 
We claim that there exists $x\in T^N$ such that 
\begin{equation}\label{eq:limit}
  \lim_{t\to T^+(\ga(x))} H\big(\vphi\big(\ga(x),t)\big)
	  \le \sup_{\ga\simeq\ga_\Si} \min_{T^N} H\circ\ga =: c^*
		\le \max H(\cL_\Si).
\end{equation}
Arguing by contradiction and using the compactness of $T^N$, we assume that there exists $\de>0$ such that  $\lim_{t\to T^+(\ga(x))} H\big(\vphi\big(\ga(x),t)\big) > c^*+\de$ for all $x\in T^N$. Now we consider the function
\[
  \tau: T^N\to[0,\infty),\quad \tau(x) := \min\{t\ge0: H\big(\vphi\big(\ga(x),t\big)\big) \ge c^*+\de\big\}.
\]
This is well defined by \eqref{eq:limit}, and it is continuous because $H\big(\vphi\big(\ga(x),t\big)\big)$ is strictly increasing in $t$. Then
\[
  \ga_1: T^N \to \cF_N(\Si),\quad \ga_1(x) := \vphi\big(\ga(x),\tau(x)\big),
\] 
is continuous and homotopic to $\ga$, hence to $\ga_\Si$, using the homotopy
\[
  T^N\times [0,1] \to \cF_N(\Si),\quad (x,t) \mapsto \vphi\big(\ga(x),t\tau(x)\big).
\]
However, $\min H\circ\ga_1 \ge c^*+\de$, contradicting \eqref{eq:linking}. This proves \eqref{eq:limit}, so that Theorem~\ref{thm:main} follows from Lemma~\ref{lem:flow limit}.
\end{altproof}

\begin{Rem}\label{rem:main-proof}
  It is easy to see that $c^* = \sup_{\ga\simeq\ga_\Si} \min_{T^N} H\circ\ga$ is a critical value as claimed in Remark~\ref{rem:main}~a). For all $\eps>0$ there exists $\ga\simeq\ga_\Si$ with $\min_{T^N} H\circ\ga > c^*-\eps$. The proof of Theorem~\ref{thm:main} shows that there exists $x\in T^N$ such that
	\[
	  c^*-\eps < H\big(\vphi(\ga(x),t)\big) \le c^*
		  \quad \text{for all $0 \le t < T^+\big(\ga(x)\big)$.}
	\]
	This implies that there exists a critical point $p_\eps\in\cF_N\Si$ of $H$ with $c^*-\eps<H(p_\eps)\le c^*$. Then $p_\eps \to p^* \in \cF_N\Si$ as $\eps \to 0$ along a subsequence. Clearly $p^*$ is a critical point of $H$ at the level $c^*$, and $Crit(H,c^*)$ is compact by Lemma~\ref{lem:collision}.
\end{Rem}

\begin{altproof}{Theorem~\ref{thm:symmetry} a)}
Since $\tau$ is an isometry the Green function is invariant under $\tau$, i.e.\ $G\big(\tau(p),\tau(q)\big)=G(p,q)$. This implies that $H$ is invariant under the induced action of $\tau$ on $\cF_3\Si$. The fixed point set of this action is $(\cF_3\Si)^\tau = \cF_3(\Si^\tau)$. Then the principle of symmetric criticality implies that a critical point of $H|_{\cF_3(\Si^\tau)}$ is also a critical point of $H$. Since $S\subset\Si^\tau$ is a connected component a critical point of $H|_{\cF_3S}$ is a critical point of $H$. According to \cite[Section~2]{Palais:1979} the submanifold $S$ is totally geodesic, hence there exists a periodic covering map $\pi:\R\to S$ which is a closed geodesic with period $L$ being the length of $S$. This implies that $d_g\big(\pi(s),\pi(t)\big) = |s-t|$ if $|s-t| \le \frac{L}2$. 

Now we consider the configuration $p^s:=\big(\pi(-s),\pi(0),\pi(s)\big) \in\cF_3S\subset\cF_3\Si$ for $s>0$ small and observe that
\[
\begin{aligned}
  H(p^s) &= -\frac1{\pi}\big(\Ga_1\Ga_2\log d_g\big(p^s_1,p^s_2\big) + \Ga_1\Ga_3\log d_g\big(p^s_1,p^s_3\big)
                       + \Ga_2\Ga_3\log d_g\big(p^s_2,p^s_3\big)\big) +O(1)\\
             &= -\frac1{\pi}(\Ga_1\Ga_2+\Ga_1\Ga_3+\Ga_2\Ga_3)\log s + O(1)\\
             &\to \infty\qquad\text{as $s\to0$}
\end{aligned}
\]
because $\Ga_1\Ga_2+\Ga_1\Ga_3+\Ga_2\Ga_3>0$ by assumption. Similarly we obtain for $s<\frac{L}2$ close to $\frac{L}2$:
\[
  H(p^s) = -\frac1{2\pi}\Ga_1\Ga_3\log\big({\textstyle \frac{L}2}-s\big) + O(1)
	       \to \infty\qquad\text{as $\textstyle s\to\frac{L}2$}
\]
because $\sign\Ga_i=(-1)^i$ by assumption. This sign condition also implies that
\[
  \textstyle \al := \sup\big\{H(p): p\in \cF_3S,\ d_g(p_1,p_3)=\frac{L}2\big\} < \infty,
\]
because fixing $p_1,p_3 \in S$ with $d_g(p_1,p_3)=\frac L2$, we have $H(p_1,p_2,p_3)\to-\infty$ as $p_2\to p_1$ or $p_2\to p_3$. Now we choose $0<\eps<\frac{L}4$ small so that $H(p^s) > \al$  for $0 < s \le \eps$ and for $\frac{L}2-\eps \le s < \frac{L}2$. Setting
\[
  \Ga := \left\{\ga:[0,1]\to\cF_3S: \ga\text{ is continuous},\ \ga(0) = p^\eps,\ \ga(1) = p^{\frac{L}2-\eps}\right\}
\]
we claim that 
\[
  c^*:= \sup_{\ga\in\Ga}\min_{0\le t\le 1} H\big(\ga(t)\big) \le\al.
\]
Given $\ga\in\Ga$ let $w:[0,1]\to\R^3$ be a lift of $\ga$, i.e.\ $\pi\circ w_i = \ga_i$ for $i=1,2,3$, such that $w(0) = (-\eps,0,\eps)$, hence $w_3(0)-w_1(0) = 2\eps < \frac{L}2$. From $w_1(t) < w_2(t) < w_3(t)$ for all $t\in[0,1]$ and $w_1(1) = -\frac{L}2+\eps$ mod $L$,  $w_3(1) = \frac{L}2-\eps$ mod $L$ we obtain that $w_3(1)-w_1(1) \ge L-2\eps > \frac{L}2$. Consequently $w_3(t)-w_1(t)=\frac{L}2$ for some $t\in[0,1]$. This implies that $d_g\big(\ga_1(t),\ga_3(t)\big) = \frac{L}2$, hence $H\big(\ga(t)\big) \le\al$. Now one can argue as in the proof of Theorem~\ref{thm:main} and Remark~\ref{rem:main-proof} in order to obtain a critical point $p^*\in\cF_3S$ of $H$ at the level $c^*$.
\end{altproof}

\begin{altproof}{Theorem~\ref{thm:symmetry} b)}
First observe that the assumptions $\Ga_1\Ga_2+\Ga_1\Ga_3+\Ga_2\Ga_3>0$ and $\Ga_1=\Ga_3<0<\Ga_2$ imply $\Ga_1+2\Ga_2 < 0$. We consider the involution
\[
  \si:\cF_3\Si \to \cF_3\Si,\quad\si(p_1,p_2,p_3) = \big(\tau(p_3),\tau(p_2),\tau(p_1)\big)
\]
with fixed point set
\[
  F := (\cF_3\Si)^\si = \{(p_1,p_2,p_3):p_2\in \Si^\tau,\ p_3=\tau(p_1)\notin\Si^\tau\} \subset \cF_3\Si.
\]
Since $H$ is invariant under $\tau$ and $\Ga_1=\Ga_3$ it follows that $H$ is invariant under $\si$, hence it suffices to find a critical point of $H|_F$. We claim that $H(p^n) \to \infty$ if $p^n\in F$ and $p^n \to \pa\cF_3\Si$. Consider a sequence $p^n \in F$ with $p^n \to p^* \in \pa\cF_3\Si\subset\Si^3$. This implies $p^*_2\in \Si^\tau$ and $p^*_3=\tau(p^*_1)$. If $p^*_1 = p^*_3 \ne p^*_2$ then $d_g(p^n_1,p^n_3) \to 0$ and $d_g(p^n_1,p^n_2)$ is bounded away from $0$, thus $H(p^n) \to \infty$. If $p^*_1=p^*_2 \in \Si^\tau$ then $p^*_3=\tau(p^*_1) = p^*_1$. Using $d_g(p^n_3,p^n_2)=d_g\big(\tau(p^n_3),\tau(p^n_2)\big) = d_g(p^n_1,p^n_2)$, hence $d_g(p^n_1,p^n_3) \le 2d_g(p^n_1,p^n_2)$ we obtain:
\[
\begin{aligned}
  H(p^n) &= -\frac1{2\pi}\big(2\Ga_1\Ga_2\log d_g(p^n_1,p^n_2) + \Ga_1^2\log d_g(p^n_1,p^n_3)\big) + O(1)\\
    &\ge -\frac1{2\pi}\big(2\Ga_1\Ga_2\log d_g(p^n_1,p^n_2) + \Ga_1^2\log d_g(p^n_1,p^n_2)\big) + O(1)\\
    &= -\frac1{2\pi}\Ga_1\log\big(d_g(p^n_1,p^n_2)^{2\Ga_2+\Ga_1}\big) + O(1)\\
             &\to \infty
\end{aligned}  
\]
From $H(p^n) \to \infty$ as $F \ni p^n \to \pa\cF_3\Si$ it follows that $H$ is bounded from below on $F$. Then $\inf_FH$ is achieved in $F$ and is a critical value.
\end{altproof}

\begin{altproof}{Proposition~\ref{prop:generic-K}}
a) For $K\in\cC^2(\Si,\R^+)$ we set $f := \log K \in Y :=\cC^2(\Si)$ and define
\[
  H_f:\cF_N\Si\to\R,\quad
  H_f(p_1,\dots,p_N) =  \sum_{\mytop{i,j=1}{i\ne j}}^N \Ga_i\Ga_jG(p_i,p_j) +  \sum_{i=1}^N\big(\Ga_i^2h(p_i,p_i) + f(p_i)\big).
\]
Clearly the set
\[
  \cM := \{f\in Y: \text{$H_f$ is a Morse function}\}
\]
is an open subset of $Y$ because the set of critical points of $H_f$ is finite for $f\in\cM$ by Lemma~\ref{lem:collision}. It remains to prove that $\cM$ is also dense in $Y$. We first prove a local version of this. Given a coordinate chart $\vphi:U\to V$ of $\Si$ with open subsets $U\subset\R^2$ and $V\subset \Si$, and setting
\[
  H_f^\vphi:\cF_NU\to\R,\quad H_f^\vphi(x_1,\dots,x_N) = H_f\big(\vphi(x_1),\dots,\vphi(x_N)\big),
\]
we claim that
\begin{equation}\label{eq:Morse}
  Y^\vphi := \big\{f\in Y:\text{$H_f^\vphi$ is a Morse function}\big\}\quad \text{is a residual subset of $Y$,}
\end{equation}
Recall that a residual subset is a set whose complement is a countable union of closed and nowhere dense subsets. Consequently $Y^\vphi$ is a dense subset of $Y$ by Baire's category theorem. In order to see \eqref{eq:Morse} consider the map $F:\cF_NU\times Y \to \R^{2N}$ defined by
\[
  F(x,f) = \nabla H_f^\vphi(x) = \nabla H_0^\vphi(x) + \big(\nabla (f\circ\vphi)(x_1),\dots,(\nabla (f\circ\vphi)(x_N)\big)^\top
\]
where $x=(x_1,\dots,x_N)\in\cF_NU$. Clearly $F$ is of class $\cC^1$, and $\frac{\pa F}{\pa x}(x,f): \R^{2N}\to \R^{2N}$ is a Fredholm operator of index $0$ for every $(x,f)\in\cF_NU\times Y$. Moreover the derivative
\[
  DF(x,f): \R^{2N}\times Y\to \R^{2N},
\]
given by
\[
\begin{aligned}
  DF(x,f)[\xi,\phi] &= \nabla^2H_0^\vphi(x)[\xi] + \big(\nabla^2 (f\circ\vphi)(x_1)[\xi_1],\dots,(\nabla^2 (f\circ\vphi)(x_N)[\xi_N]\big)^\top\\
                         &\hspace{1cm} + \big(\nabla (\phi\circ\vphi)(x_1),\dots,(\nabla (\phi\circ\vphi)(x_N)\big)^\top
\end{aligned}
\]
is surjective: Given $v_1,\dots,v_N\in\R^2$ there exists $\phi\in Y$ such that $\nabla(\phi\circ\vphi)(x_i) = v_i$ for $i=1,\dots,N$ because $x\in\cF_NU$. Now \cite[Theorem~5.4]{Henry:2005} implies that
\[
  Y^\vphi = \{f\in Y: \text{$0$ is a regular value of $F(\,\cdot\,,f):\cF_NU\to\R^{2N}$}\}\quad\text{is a residual subset of $Y$.}
\]
 This proves \eqref{eq:Morse}.

Now we prove that $\cM$ is a dense subset of $Y$. Given $f_0\in Y$, the set
\[
  K_{f_0} := \{p\in\cF_N\Si: \nabla H_{f_0}(p)=0\}
\]
of critical points of $H_{f_0}$ is compact as a consequence of Lemma~\ref{lem:collision}. By \eqref{eq:Morse}, for each $p\in K_{f_0}$ there exist an open neighborhood $V_p\subset\cF_N\Si$ of $p$ such that
\[
  W_p := \{f\in Y: H_f|_{V_p} \text{ is a Morse function}\}\quad\text{is a residual subset of $Y$.}
\]
There exist $p^1,\dots,p^k\in K_{f_0}$ such that $K_{f_0} \subset \bigcup_{i=1}^k V_{p^i} =: V$. Then $W := \bigcap_{i=1}^k W_{p^i}$ is a dense subset of $Y$, and $H_f|_{V}$ is a Morse function for every $f\in W$. Applying Lemma~\ref{lem:collision} once more we see that $|\nabla H_{f_0}|$ is bounded away from $0$ on $\cF_N\Si\setminus  V$. It follows that $H_f$ is a Morse function for every $f\in W$ close to $f_0$.

The proof of b) proceeds analogously.
\end{altproof}

\begin{Rem}\label{rem:proof Morse}
  Setting
  \[
    H_\psi(p_1,\ldots,p_N) = \sum_{\mytop{i,j=1}{i\ne j}}^N \Ga_i\Ga_jG(p_i,p_j) + \Psi(p_1,\dots,p_N)
  \]
  for $\Psi\in Z:=\cC^2(\Si^N)$ one can argue as in the proof of Proposition~\ref{prop:generic-K} to see that
  \[
    \cN := \{\Psi\in Z: H_\Psi \text{ is a Morse function}\}
  \]
  is an open and dense subset of $Z$. First, $\cN$ is an open subset of $\cC^2(\Si^N)$ as a consequence of Lemma~\ref{lem:collision}. Secondly, for a coordinate chart $\vphi:U\to V$ of $\Si^N$ with open subsets $U\subset \R^{2N}$ and $V\subset\Si^N$ we set
  \[
    H_\Psi^\vphi := H_\Psi\circ \vphi: \vphi^{-1}(V\cap \cF_N\Si) \to \R 
  \]
  and claim that
  \[
    Z^\vphi := \{\Psi\in Z: H_\Psi^\vphi \text{ is a Morse function}\}\quad\text{is a residual subset of $Z$.}
  \]
  This follows from \cite[Theorem~5.4]{Henry:2005} applied to the map
  \[
    U \times Z \to \R^{2N},\quad (x,\Psi) \mapsto \nabla H_\Psi^\vphi = \nabla H_0^\vphi(x) + \nabla(\Psi\circ\vphi)(x).
  \]
  Finally, $\cN$ is a residual subset of $Z$ by the covering argument from the proof of Proposition~\ref{prop:generic-K}.
\end{Rem}

{\bf Acknowledgement:} The authors thank the reviewer for her/his careful reading and helpful remarks that improved the presentation of the paper.


\begin{thebibliography}{99}
\labelsep=1em\relax

\bibitem{Ahmedou-BenAyed:2017}
{\sc M.~Ahmedou and M.~Ben~Ayed},
{\em Theory of ``critical points at infinity'' and a resonant singular {L}iouville-type equation}, 
Adv. Nonlinear Stud. {\bf 17} (2017), 139--166.

\bibitem{Ahmedou-BenAyed:2021}
{\sc M.~Ahmedou and M.~Ben~Ayed},
{\em Morse inequalities at infinity for a resonant mean field equation},
Commun. Contemp. Math. (to appear).

\bibitem{Ahmedou-BenAyed-Lucia:2017}
{\sc M.~Ahmedou, M.~Ben~Ayed, and M.~Lucia},
{\em On a resonant mean field type equation: a ``critical point at infinity" approach},
 Discrete Contin. Dyn. Syst. {\bf 37} (2017), 1789--1818.

\bibitem{Aubin:1998}
{\sc T.~Aubin},
{\em Some nonlinear problems in Riemannian geometry},
Springer-Verlag, Berlin 1998.

\bibitem{Bartolucci-Pistoia:2007}
{\sc D.~Bartolucci and A.~Pistoia},
{\em Existence and qualitative properties of concentrating solutions for the sinh-Poisson equation}, 
IMA J. Appl. Math. {\bf 72} (2007, 706--729.

\bibitem{Ba:2016}
{\sc T.~Bartsch},
{\em Periodic solutions of singular first-order {H}amiltonian systems of $N$-vortex type},
Arch. Math. (Basel) {\bf 107} (2016), 413--422.

\bibitem{Ba-Dai:2016}
{\sc T.~Bartsch and Q.~Dai},
{\em Periodic solutions of the {$N$}-vortex {H}amiltonian system in planar domains},
J. Differential Equations {\bf 260} (2016), 2275--2295.

\bibitem{Ba-Gebhard:2017}
{\sc T.~Bartsch and B.~Gebhard},
{\em Global continua of periodic solutions of singular first-order {H}amiltonian systems of {N}-vortex type},
Math. Ann. {\bf 369} (2017), 627--651.

\bibitem{Ba-Pistoia:2015}
{\sc T.~Bartsch and A.~Pistoia},
{\em Critical points of the {$N$}-vortex {H}amiltonian in bounded planar domains and steady state solutions of the incompressible {E}uler equations},
SIAM J. Appl. Math. {\bf 75} (2015), 726--744.

\bibitem{Ba-Pistoia-Weth:2010}
{\sc T.~Bartsch, A.~Pistoia, and T.~Weth},
{\em $N$-vortex equilibria for ideal fluids in bounded planar domains and new nodal solutions of the sinh-Poisson and the Lane-Emden-Fowler
  equations},
Comm. Math. Phys. {\bf 297} (2010), 653--686.

\bibitem{Ba-Sacchet:2018}
{\sc T.~Bartsch and M.~Sacchet},
{\em Periodic solutions with prescribed minimal period of vortex type problems in domains},
Nonlinearity {\bf 31} (2018), 2156--2172.

\bibitem{Boatto-Koiller:2014}
{\sc S.~Boatto and J.~Koiller},
{\em Vortices on closed surfaces},
in: Geometry, mechanics, and dynamics, Fields Inst. Commun. {\bf 73}, Springer, New York 2015, 185--237.

\bibitem{chen-Lin:2003}
{\sc C.-C. Chen and C.-S. Lin},
{\em Topological degree for a mean field equation on {R}iemann surfaces},
Comm. Pure Appl. Math. {\bf 56} (2003), 1667--1727.

\bibitem{dAprile:2013}
{\sc T. D'Aprile},
{\em Multiple blow-up solutions for the Liouville equation with singular data},
Comm. Partial Differential Equations {\bf 38} (2013), 1409--1436.

\bibitem{dAprile:2015}
{\sc T. D'Aprile},
{\em Sign-changing blow-up solutions for H\'enon type elliptic equations},
J. Funct. Anal. {\bf 268} (2015), 2067--2101.

\bibitem{dAprile-Esposito:2017}
{\sc T.~D'Aprile and P.~Esposito},
{\em Equilibria of point-vortices on closed surfaces},
Ann. Sc. Norm. Super. Pisa Cl. Sci. (5) {\bf 17} (2017), 287--321.

\bibitem{DelPino-et-al:2015}
{\sc M.~del Pino, P.~Esposito, P.~Figueroa, and M.~Musso},
{\em Nontopological condensates for the self-dual {C}hern-{S}imons-{H}iggs model},
Comm. Pure Appl. Math. {\bf 68} (2015), 1191--1283.

\bibitem{delpino-Kowalczyk-Musso:2005}
{\sc M.~del Pino, M.~Kowalczyk, and M.~Musso},
{\em Singular limits in {L}iouville-type equations},
Calc. Var. Partial Differential Equations {\bf 24} (2005), 47--81.

\bibitem{Dritschel-Boatto:2015}
{\sc D.~G. Dritschel and S.~Boatto},
{\em The motion of point vortices on closed surfaces},
Proc. A. {\bf 471} (2015), no. 2176, 20140890, 25 pp.

\bibitem{Esposito-Figueroa:2014}
{\sc P.~Esposito and P.~Figueroa},
{\em Singular mean field equations on compact {R}iemann surfaces},
Nonlinear Anal. {\bf 111} (2014), 33--65.

\bibitem{Figueroa:2014}
{\sc P.~Figueroa},
{\em Singular limits for Liouville equations on the flat two-torus},
Calc. Var. Partial differential Equations {\bf 49} (2014), 613--647.

\bibitem{Figueroa:2022}
{\sc P.~Figueroa},
{\em Bubbling solutions for mean field equations with variable intensities on compact {R}iemann surfaces},
 arXiv:2203.09731, 2022.

\bibitem{Gebhard:2018}
{\sc B.~Gebhard},
{\em Periodic solutions for the {$N$}-vortex problem via a superposition principle},
Discrete Contin. Dyn. Syst. {\bf 38} (2018), 5443--5460.

\bibitem{Helmholtz:1858} 
{\sc H. Helmholtz},
{\em \"{U}ber Integrale der hydrodynamischen Gleichungen, welche den Wirbelbewegungen entsprechen},
J. reine angew. Math. {\bf 55} (1858), 25--55.

\bibitem{Henry:2005}
{\sc D.~Henry},
{\em Perturbation of the boundary in boundary-value problems of partial differential equations}, 
volume 318 of London Mathematical  Society Lecture Note Series, Cambridge University Press, Cambridge, 2005.
 With editorial assistance from Jack Hale and Ant\^{o}nio Luiz Pereira.

\bibitem{Kimura:1999}
{\sc Y.~Kimura},
{\em  Vortex motion on surfaces with constant curvature},
R. Soc. Lond. Proc. Ser. A Math. Phys. Eng. Sci. {\bf 455} (1999), 245--259.

\bibitem{Kirchhoff:1876} 
{\sc G. Kirchhoff},
{\em Vorlesungen \"{u}ber mathematische Physik.}
Teubner, Leipzig 1876

\bibitem{Kuhl:2015}
{\sc C.~Kuhl},
{\em Symmetric equilibria for the {$N$}-vortex problem},
J. Fixed Point Theory Appl. {\bf 17} (2015), 597--624.

\bibitem{Kuhl:2016}
{\sc C.~Kuhl},
{\em Equilibria for the {$N$}-vortex-problem in a general bounded domain},
J. Math. Anal. Appl. {\bf 433} (2016), 1531--1560.

\bibitem{Majda-Bertozzi:2002} 
{\sc A.~J.~Majda, A.~L.~Bertozzi},
{\em Vorticity and Incompressible Flow},
Cambridge University Press 2002.

\bibitem{Newton:2001}
{\sc P.~K. Newton},
{\em The $N$-vortex problem},
 volume 145 of Applied Mathematical Sciences, Springer-Verlag, New York, 2001.

\bibitem{Palais:1979}
{\sc R.~S. Palais},
{\em The principle of symmetric criticality},
Comm. Math. Phys. {\bf 69} (1979), 19--30.

\bibitem{Saffman:1992} 
{\sc P.~G.~Saffman}, 
{\em Vortex dynamics},
Cambridge University Press, Cambridge 1992.

\bibitem{Sakajo-Shimizu:2016}
{\sc T.~Sakajo and Y.~Shimizu},
{\em Point vortex interactions on a toroidal surface},
Proc. A. {\bf 472} (2016), no. 2191, 20160271, 24 pp.

\end{thebibliography}

{\sc Address of the authors:}\\

 \noindent
 Mathematisches Institut \\
 University of Giessen \\
 Arndtstr.\ 2 \\
 35392 Giessen \\
 Germany \\
 Mohameden.Ahmedou@math.uni-giessen.de\\
 Thomas.Bartsch@math.uni-giessen.de \\
 Tim.Fiernkranz@math.uni-giessen.de

\end{document}